\documentclass[11pt]{amsart}
\usepackage{amssymb,amsmath,mathrsfs}
\usepackage{enumerate,tikz}
\usepackage{hyperref,soul}
\usepackage{xstring}
\usetikzlibrary{patterns}

\newtheorem{thm}{Theorem}[section]
\newtheorem{prop}[thm]{Proposition}
\newtheorem{lem}[thm]{Lemma}
\newtheorem{conjecture}[thm]{Conjecture}

\theoremstyle{definition}
\newtheorem*{rem}{Remark}

\numberwithin{equation}{section}

\usepackage{mathtools}

\DeclarePairedDelimiter\floor{\lfloor}{\rfloor}

\definecolor{red1}{HTML}{B02400}

\def\rcol#1{{\color{red1}#1}}
\def\bcol#1{{\color{blue}#1}}
\def\abs#1{\lvert#1\rvert}

\def\sif{{SIF}\;}
\def\Sif{\mathcal{S}^{\hskip0.7pt\mathsf{sif}}}
\def\Ind{\mathcal{S}^{\hskip0.7pt\mathsf{ind}}}
\def\Der{\mathcal{S}^{\hskip0.5pt\mathsf{der}}}

\def\oeis#1{\cite[#1]{oeis}}
\DeclareMathOperator\red{red}
\DeclareMathOperator\id{id}

\def\R{\rule[-1ex]{0ex}{3.6ex}}

\tikzstyle{mesh}=[pattern=north east lines, pattern color=gray, draw=gray]

\newcommand{\plotpermutation}[1]{%
\def\Permutation{#1}

\StrDel{\Permutation}{,}[\OnelineP]
\StrLen{\OnelineP}[\N]

\def\rad{0.15}
\begin{scope}
\clip (0.5,0.5) rectangle (\N.5,\N.5);
\draw[gray!60] (0,0) grid (\N+1,\N+1);
\foreach [count=\i] \y in \Permutation {%
  \draw[fill] (\i,\y) circle (\rad);
  }
\end{scope}
}

\begin{document}
\title{Pattern-avoiding stabilized-interval-free permutations}
\author{Daniel Birmajer}
\address{Nazareth College\\ 4245 East Ave.\\ Rochester, NY 14618}
\email{abirmaj6@naz.edu}

\author[Juan Gil]{Juan B. Gil}
\address{Penn State Altoona\\ 3000 Ivyside Park\\ Altoona, PA 16601}
\email{jgil@psu.edu}

\author[Jordan Tirrell]{Jordan O. Tirrell}
\address{Washington College\\ 300 Washington Avenue\\ Chestertown, MD 21620}
\email{jtirrell2@washcoll.edu}

\author[Michael Weiner]{Michael D. Weiner}
\address{Penn State Altoona\\ 3000 Ivyside Park\\ Altoona, PA 16601}
\email{mdw8@psu.edu}


\begin{abstract}
In this paper, we study pattern avoidance for stabilized-interval-free (SIF) permutations. These permutations are contained in the set of indecomposable permutations and in the set of derangements. We enumerate pattern-avoiding \sif permutations for classical and pairs of patterns of size 3. In particular, for the patterns $123$ and $231$, we rely on combinatorial arguments and the fixed-point distribution of general permutations avoiding these patterns. We briefly discuss $123$-avoiding permutations with two fixed points and offer a conjecture for their enumeration by the distance between their fixed points. For the pattern $231$, we also give a direct argument that uses a bijection to ordered forests.
\end{abstract}

\maketitle

\section{Introduction}

In 2004 D.~Callan \cite{Callan04} introduced the notion of a {\em stabilized-interval-free} (SIF) permutation. Let $\sigma$ be a permutation on $[n]=\{1,\dots,n\}$ and let $[i,j]$ be an interval. Recall that $i$ is a fixed point of $\sigma$ if $\sigma(i)=i$. Moreover, $[i,j]$ is a stabilized interval of $\sigma$ if $\sigma(k)\in[i,j]$ for every $k\in[i,j]$. The permutation $\sigma$ is called \sif if it does not stabilize any proper subinterval of $[n]$. For example, the permutation $562314$ (one-line notation) is \sif but $634215$ is not since $[2,4]$ is a stabilized interval. Observe that the reduced permutation\footnote{$\red(\sigma)$ is obtained by replacing the $i$th smallest letter of $\sigma$ by $i$, for $i=1,\dots,|\sigma|$.} $\red(342)=231$ is SIF. Such a minimal component will be referred to as a \sif box:

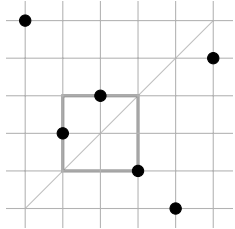
\smallskip
\begin{center}
\begin{tikzpicture}[scale=0.5]
\draw[gray!50] (1,1)--(6,6);
\draw[very thick,gray!80] (2,2) rectangle (4,4);
\plotpermutation{6,3,4,2,1,5}
\end{tikzpicture}
\end{center}
A fixed point is a \sif box of size 1, and adjacent transpositions are \sif boxes of size 2. In general, a \sif box of size $m$ in a permutation is a stabilized interval of length $m$ whose reduced permutation is SIF.

As discussed in \cite{Callan04}, \sif permutations can be used as building blocks to decompose permutations into noncrossing components and are therefore naturally related to other combinatorial families structured by noncrossing partitions. For example, F.~Ardila et al.\ \cite{ARW16} showed that \sif permutations are in bijection to connected positroids, and N.~Blitvi{\'c} \cite{NB14} introduced the set of chord-connected permutations which, in a similar way, can be used as building blocks for \sif permutations.

Another related family of permutations are the {\em indecomposable} permutations studied by Comtet \cite{Comtet72} in 1972. These are permutations that do not stabilize a proper subinterval of $[n]$ that includes $1$. Denoting the above sets by $\Sif_n$ and $\Ind_n$, respectively, and letting $\mathcal{S}_n$ denote the set of permutations on $[n]$, we have $\Sif_n \subset \Ind_n \subset \mathcal{S}_n$. Moreover, $\Sif_n$ is contained in the set $\Der_n$ of fixed-point-free permutations (a.k.a.\ derangements) of size $n$. 
 
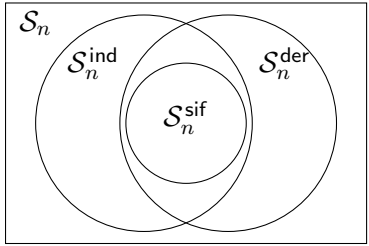
\begin{center}
\begin{tikzpicture}[scale=0.8]
\draw (0,0) rectangle (6,4);
\draw (2.3,2) circle(1.8);
\draw (3.7,2) circle(1.8);
\draw (3,2) circle (1);
\node[above=1] at (0.5,3.25) {$\mathcal{S}_n$};
\node[left=1] at (2.1,3) {$\Ind_n$};
\node[right=1] at (4,3) {$\Der_n$};
\node at (3,2.1) {$\Sif_n$};
\end{tikzpicture}
\end{center}

In this paper, we enumerate pattern-avoiding \sif permutations for classical and pairs of patterns of size 3. Several authors have studied pattern avoidance for indecomposable permutations (see e.g.\ \cite{GKZ16}) and for derangements (see \cite{RSZ03,Eli04}), but there seems to be no results in the literature pertaining to \sif permutations. For basic definitions and standard notation regarding permutation patterns, we refer the reader to the book by Kitaev \cite{Kitaev11}. Let us recall that for a permutation $\sigma = \sigma(1)\cdots\sigma(n)$, its reverse is the permutation $r(\sigma)=\sigma(n)\cdots\sigma(1)$ and its complement is the permutation $c(\sigma)=(n+1-\sigma(1))\cdots (n+1-\sigma(n))$. We let $\sigma^{-1}$ denote the inverse of $\sigma$ and $\sigma^{rc}=r(c(\sigma))$.

We denote by $\Sif_n(\sigma)$ the set of \sif permutations of size $n$ that avoid the pattern $\sigma$. Note that stabilized intervals are preserved under the inverse and reverse complement maps. As a consequence, we have the basic Wilf equivalence relations
\begin{equation} \label{eq:basic_equivalence}
\Sif_n(\sigma) \sim \Sif_n(\sigma^{-1}) \sim \Sif_n(\sigma^{rc}).
\end{equation}

For patterns of size 3, there are three equivalence classes (see Table~\ref{tab:3Patterns}), and in Section~\ref{sec:3-pattern}, we give enumerating formulas for $\Sif_n(\sigma)$ for every $\sigma\in\mathcal{S}_3$. For the enumeration of $\Sif_n(123)$ and $\Sif_n(231)$, we rely on combinatorial arguments together with known results on the fixed-points distribution of permutations in $\mathcal{S}_n(123)$ and $\mathcal{S}_n(231)$, respectively. Regarding $\Sif_n(231)$, we also give a direct argument that uses a bijection to ordered forests. In Section~\ref{sec:pair_of_patterns}, we examine pairs of patterns of size 3 and obtain formulas for the nonfinite sequences given in Table~\ref{tab:PairOfPatterns}. Surprisingly, the enumeration of $\Sif_n(123,231)$ led to a non-monotonic sequence for which we give a generating function and closed formulas.

Finally, given their connection to the set $\Sif_n(123)$, we revisit the set of permutations in $\mathcal{S}_n(123)$ having two fixed points. In the appendix, we organize such permutations by the distance between their fixed points and conjecture a new formula for their enumeration. 

\section{Avoiding a pattern of size 3}
\label{sec:3-pattern}

As for pattern-avoiding fixed-point-free permutations, there are three Wilf equivalence classes of \sif permutations avoiding a single pattern of size 3 (see Table~\ref{tab:3Patterns}).

\begin{table}[ht]
\small
\begin{tabular}{|c|l|c|c|} \hline
\R $\sigma\in\mathcal{S}_3$ & \hspace{80pt} $\abs{\Sif_n(\sigma)}$, $n\ge 1$  & OEIS & Section \\[2pt] \hline
\rule[-4ex]{0ex}{9.2ex} \parbox{4ex}{132\\ 213\\ 321} & 1, 1, 2, 5, 14, 42, 132, 429, 1430, 4862, 16796, 58786, \dots & A000108 & 2.1\\ \hline
\R 123 & 1, 1, 2, 5, 14, 44, 150, 496, 1758, 6018, 21782, 76414, \dots & A363431 & 2.2 \\ \hline
\rule[-2.5ex]{0ex}{6.1ex} \parbox{4ex}{231\\ 312} & 1, 1, 1, 2, 6, 18, 54, 170, 551, 1817, 6092, 20722, \dots & A363432 & 2.3 \\ \hline
\end{tabular}
\bigskip
\caption{$\sigma$-avoiding \sif permutations.}
\label{tab:3Patterns}
\end{table}

Before we proceed, we would like to highlight the fact that every \sif permutation $\tau$ of size $n>2$ must contain a $231$ or a $312$ pattern. Otherwise, $\tau$ would be in $S_n(231,312)$, the set of layered permutations\footnote{A permutation is called layered if it consists of a disjoint union of factors so that the letters decrease within each layer, and increase between the layers, \cite[Definition~2.1.11]{Kitaev11}.}. A layered permutation of size $n>2$ is never \sif because it is either decomposable or it is the permutation $n(n-1)\cdots21$, which has either a fixed point (when $n$ is odd) or a \sif box of size 2 (when $n$ is even).

Let us also recall the definitions of skew and direct sums of two permutations. The skew sum $\pi_1\ominus\pi_2$ is the permutation obtained by listing the elements of $\pi_1$, each increased by the size of $\pi_2$, followed by the elements of $\pi_2$. The direct sum $\pi_1\oplus\pi_2$ is the permutation obtained by listing the elements of $\pi_1$, followed by the elements of $\pi_2$ increased by the size of $\pi_1$. For example, $231\ominus 12 = 45312$ and $231\oplus 21 = 23154$.

\subsection{Enumeration of $\Sif_n(132)\sim \Sif_n(213)\sim \Sif_n(321)$}
We will start with the simple case of \sif permutations avoiding $321$ and will then discuss the enumeration of $\Sif_n(132)$. The equivalence to $\Sif_n(213)$ follows from \eqref{eq:basic_equivalence}. 

Let $C_n=\frac{1}{n+1}\binom{2n}{n}$ and $C(x)=\frac{1-\sqrt{1-4x}}{2x}=\sum\limits_{n=0}^\infty C_n x^n$.

\begin{thm} \label{thm:321avoiders}
For $n\ge 1$, we have $\Sif_n(321) = \Ind_n(321)$. Therefore, $\abs{\Sif_n(321)} = C_{n-1}$.
\end{thm}
\begin{proof}
By definition, $\Sif_n(321)\subset \Ind_n(321)$. Suppose $\sigma\in\Ind_n(321)$ has a stabilized interval $[j,k]$, $1<j<k<n$. Since $\sigma$ is indecomposable, neither $[1,j-1]$ nor $[k+1,n]$ can be stabilized intervals, so there must exist $i$ and $\ell$ with $1\le i< j<k<\ell \le n$ and such that $\sigma(i)>k$ and $\sigma(\ell)<j$. This would create a $321$ pattern $(\sigma(i),j,\sigma(\ell))$, giving a contradiction. In other words, $\Ind_n(321)\subset \Sif_n(321)$ and we get equality. The fact that $321$-avoiding indecomposable permutations are counted by the Catalan numbers is known and not difficult to prove, see e.g.\ \cite[Theorem~3.2]{GKZ16}.
\end{proof}

\begin{thm}
For $n\ge 1$, we have $\abs{\Sif_n(132)} = C_{n-1}$.
\end{thm}
\begin{proof}
Any permutation $\sigma\in\mathcal{S}_n(132)$ with $\sigma(n)<n$ must be of the form $\sigma=\sigma_L\ominus \sigma_R$ and is therefore indecomposable. On the other hand, there are $C_{n-1}$ permutations in $\mathcal{S}_n(132)$ with $\sigma(n)=n$, thus $\abs{\Ind_n(132)} = C_n - C_{n-1}$.

For every $k\in\{1,\dots,n-2\}$, we define a bijective map $\phi$ between pairs of permutations in $\Sif_k(132)\times\Ind_{n-k}(132)$ and the set of permutations in $\Ind_{n}(132) \setminus \Sif_{n}(132)$ having a left-most \sif box of size $k$, as follows:

Given $(\rho,\tau)\in\Sif_k(132)\times\Ind_{n-k}(132)$, let $j$ be the smallest value for which 
\[ \min\{\tau(1),\dots,\tau(j)\}\le j. \]
Geometrically, $j$ is the minimum value such that $[1,j]\times[1,j]$ intersects the plot of $\tau$. Thus, the permutation $\tau(1)\cdots\tau(j-1)$ has no stabilized interval or fixed point, and there is no $21$ pattern in the sector $[j,n]\times[j,n]$. Now, inserting $j$ into $\tau$ as a fixed point, and inflating $j$ with the permutation $\rho$, gives a permutation $\phi(\rho,\tau)$ in $\Ind_{n}(132) \setminus \Sif_{n}(132)$ with a left-most \sif box of size $k$. For example, the permutation 894532167 can be created by inserting the \sif permutation $231$ into $562134$ at position $j=3$ as illustrated in Figure~\ref{fig:sif132}.
\begin{figure}[ht]
\begin{tikzpicture}[scale=0.5]
\begin{scope}
\plotpermutation{5,6,2,1,3,4}
\foreach \i in {2,3}{\draw[red1,dashed] (1,\i)--(\i,\i)--(\i,1);}
\draw[fill,red1] (1,1) circle (0.02);
\end{scope}
\begin{scope}[xshift=240]
\plotpermutation{6,7,3,2,1,4,5}
\draw[fill,red1] (3,3) circle (\rad);
\draw[gray,thick] (3,3) circle (\rad+0.09);
\end{scope}
\begin{scope}[xshift=500]
\plotpermutation{8,9,4,5,3,2,1,6,7}
\draw[very thick,gray!80] (3,3) rectangle (5,5);
\foreach \i/\j in {3/4,4/5,5/3}{\draw[fill,red1] (\i,\j) circle (\rad);}
\end{scope}
\end{tikzpicture}
\caption{Construction $562134 \mapsto 67\rcol{3}2145 \mapsto 89\rcol{453}2167$.}
\label{fig:sif132}
\end{figure}

Observe that by our definition of $j$, if $i<j$, then $\min\{\tau(1),\dots,\tau(i)\}>i$ and $i=\tau(k)$ for some $k>i$. Therefore, inserting $i$ into $\tau$ as a fixed point would create the $132$ pattern $(i,\tau(i)+1,i+1)$. In other words, $j$ is the first position where $\rho$ can be inserted into $\tau$ to create an element of $\Ind_{n}(132) \setminus \Sif_{n}(132)$. Now, given a permutation $\sigma$ in the target set with a left-most \sif box $\rho$ of size $k$, we let $\phi^{-1}(\sigma)$ be the pair $(\red(\rho),\red(\tau))$, where $\tau$ is obtained from $\sigma$ be removing $\rho$.

As a consequence of the above bijection, we have 
\[ \abs{\Ind_n(132)} = \abs{\Sif_n(132)}  + \sum_{k=1}^{n-2}\, \abs{\Sif_k(132)}\cdot \abs{\Ind_{n-k}(132)}, \] 
which leads to the equation $I_{132}(x) = A_{132}(x) + (I_{132}(x)-x)A_{132}(x)$ for their generating functions
$I_{132}(x)=\sum\limits_{n=1}^\infty\, \abs{\Ind_n(132)}x^n$ and $A_{132}(x)=\sum\limits_{n=1}^\infty\, \abs{\Sif_n(132)}x^n$.

Finally, since $I_{132}(x) = (1-x)(C(x)-1)$ (see \cite[Theorem 3.7]{GKZ16}), we get
\[ A_{132}(x) = \frac{I_{132}(x)}{1-x+I_{132}(x)} = xC(x), \]
thus $\abs{\Sif_n(132)} = C_{n-1}$.
\end{proof}

\begin{rem}
It turns out that the Knuth bijection from $\mathcal{S}_n(321)$ to $\mathcal{S}_n(132)$ (via RSK and Dyck paths, see \cite[Sec.~4.1.1]{Kitaev11}) bijectively maps $\Sif_n(321)$ to $\Sif_n(132)$. On the other hand, letting $\beta:\mathcal{S}_n(132)\to \mathcal{S}_n(123)$ be the bijective map given by Simion and Schmidt in \cite[Prop.~19]{SiSch85}, the composition $r\circ \beta$ of the reverse map with $\beta$ seems to bijectively map $\mathcal{S}_n(132)\setminus\Ind_{n}(132)$ to $\Sif_n(321)$. In particular, $\sigma\in\Sif_n(132)$ implies $r(\beta(\sigma))\not\in\Sif_n(321)$. It would be interesting to explore how the other known bijections between $\mathcal{S}_n(132)$ and $\mathcal{S}_n(321)$ handle the \sif components of the two sets.  
\end{rem}

\subsection{Enumeration of $\Sif_n(123)$}

In order to count the elements of $\Sif_n(123)$, we rely on the enumeration of $123$-avoiding permutations by the number of fixed points. Note that permutations in $\mathcal{S}_n(123)$ can have at most two fixed points. Let $\mathrm{fp}(\sigma)$ denote the number of fixed points of $\sigma$. For $k\in\{0,1,2\}$, let 
\[ f^{(k)}_n(123) = \abs{\{\sigma\in\mathcal{S}_n(123): \mathrm{fp}(\sigma)=k\}}. \]

\begin{prop}[{\cite[Corollary~3.2]{Eli04}}]
For $n\ge 1$, we have
\begin{align*}
  f^{(0)}_n(123) &=
  \begin{cases}
  C_n - 2C_{n-1} + f^{(2)}_n(123) &\text{if $n$ is even,} \\
  C_n - 2C_{n-1} + f^{(2)}_n(123) + C_{\frac{n-1}{2}}^2 &\text{if $n$ is odd,}
  \end{cases} 
  \\[3pt]
  f^{(1)}_n(123) &=
  \begin{cases}
  2\big(C_{n-1} - f^{(2)}_n(123)\big) &\text{if $n$ is even,} \\
  2\big(C_{n-1} - f^{(2)}_n(123)\big) - C_{\frac{n-1}{2}}^2 &\text{if $n$ is odd.}
  \end{cases}
\end{align*}
\end{prop}
For $f^{(2)}_n(123)$, S.~Elizalde \cite[Theorem~3.3]{Eli04} gives a formula that we do not need to include here. In the appendix, we approach the counting of permutations in $\mathcal{S}_n(123)$ with two fixed points from a different perspective and conjecture a simpler formula for their enumeration.

\medskip
For $2\le m\le n$, let $\mathcal{B}_n^{(m)}(123)$ be the set of fixed-point-free permutations $\sigma\in\mathcal{S}_n(123)$ such that $\sigma$ contains a smallest \sif box of size $m$. Let $b_n^{(m)}(123) = \abs{\mathcal{B}_n^{(m)}(123)}$.

\begin{lem} \label{lem:2box}
For $n\ge 4$,
\begin{equation*}
 b_n^{(2)}(123) = f_{n-1}^{(1)}(123) - f_{n-2}^{(2)}(123).
\end{equation*}
\end{lem}
\begin{proof}
Every permutation $\sigma\in \mathcal{S}_{n-1}(123)$ having exactly one fixed point generates a permutation $\hat\sigma\in \mathcal{B}_n^{(2)}(123)$, created by inflating the fixed point with a $21$ box: If $\sigma(j)=j$, replace $j$ by $(j+1)j$, and relabel every value $k$ in $\sigma$ bigger than $j$ (using $k \mapsto k+1$ if  $k>j$).
For example, $54\rcol{3}12\mapsto 65\bcol{43}12$. This map is surjective, but not injective. For instance:
\begin{align*}
    5\rcol{2}431 &\mapsto 6\bcol{32}541, \\
    532\rcol{4}1 &\mapsto 632\bcol{54}1.
\end{align*}
The elements counted twice correspond to those that can be obtained from permutations in $\mathcal{S}_{n-2}(123)$ having two fixed points. For example:
  \begin{equation*}
    4\rcol{2}31 \mapsto 5\bcol{32}\rcol{4}1 \mapsto 6\bcol{3254}1.
  \end{equation*}
By inclusion-exclusion, we get the claimed formula.
\end{proof}

\begin{thm}\label{thm:123avoiders}
Let $a_n(123) = \abs{\Sif_n(123)}$. We have $a_1(123)=a_2(123)=1$, and for $n\ge 3$,
\begin{equation*}
 a_n(123) = f_n^{(0)}(123) - f_{n-1}^{(1)}(123) + f_{n-2}^{(2)}(123) - \sum_{k=1}^{\floor{\frac{n-3}{2}}} C_k^2\, a_{n-2k}(123).
\end{equation*}
This implies
\[ A_{123}(x) = F_0(x) - xF_1(x) + x^2F_2(x) - Q(x^2)(A_{123}(x) - x - x^2), \]
where $A_{123}(x)=\sum\limits_{n=1}^\infty a_n(123)x^n$, $F_k(x)=\sum\limits_{n=1}^\infty f_n^{(k)}(123) x^n$, and $Q(x) = \sum\limits_{n=0}^\infty C_n^2\, x^n$.
\end{thm}
\begin{proof}
We focus on the $123$-avoiding fixed-point-free permutations that contain a \sif box and aim at deriving a formula for $f_n^{(0)}(123) - a_n(123)$. By Lemma~\ref{lem:2box}, we already know how to count the elements of $\mathcal{B}_n^{(2)}(123)$. Suppose $\sigma\in \mathcal{B}_n^{(m)}(123)$ for some $m>2$. In this case, there is only one \sif box of size $m$, and it must contain an ascent.\footnote{Since the permutation $m(m-1)\cdots 1$ is guaranteed to have either a fixed point or a $21$ \sif box.} Therefore, $\sigma$ must be of the form $\sigma = \sigma_1\ominus \beta\ominus \sigma_2$, where $\beta\in \Sif_m(123)$ and $\sigma_1,\sigma_2$ are $123$-avoiders with $\abs{\sigma_1} = \abs{\sigma_2}$. It follows that $n-m = 2k$ for some $k\in\{1,\dots,\floor{\frac{n-3}{2}}\}$. Since $\sigma_1,\sigma_2\in\mathcal{S}_k(123)$ can be arbitrary, there are $C_k^2$ such pairs, and we have
\[ b_n^{(m)}(123) = b_n^{(n-2k)}(123) =  C_k^2\cdot a_{n-2k}(123) \;\text{ for } n-2k \ge 3. \]
In conclusion,
\begin{align*} 
  f_n^{(0)}(123) - a_n(123) &= b_n^{(2)}(123) + \sum_{k=1}^{\floor{\frac{n-3}{2}}} b_n^{(n-2k)}(123) \\
  &= f_{n-1}^{(1)}(123) - f_{n-2}^{(2)}(123) + \sum_{k=1}^{\floor{\frac{n-3}{2}}}C_k^2\, a_{n-2k}(123),
\end{align*}
which leads to the claimed formula for $a_n(123)$.
\end{proof}

The function $Q(x)$ is the generating function of the sequence of squares of Catalan numbers, which can be found in \oeis{A001246}.

\subsection{Enumeration of $\Sif_n(231)\sim \Sif_n(312)$}

Let $\mathcal{T}_n(231)$ be the set of all permutations in $\mathcal{S}_n(231)$ that are either \sif or sum decomposable into \sif permutations of size greater than or equal to 3. We have $\mathcal{T}_n(231)=\Sif_n(231)$ for $n\le 5$, and $\mathcal{T}_6(231)$ consists of $\Sif_6(231)$ together with $312645=312\oplus 312$. 

By definition, $t_n(231)=\abs{\mathcal{T}_n(231)}$ is the {\sc invert} transform of the sequence $(a_n)_{n\in\mathbb{N}}$ given by $a_1=a_2=0$ and $a_n = \abs{\Sif_n(231)}$ for $n\ge 3$. Thus their generating functions satisfy
\[ 1 + T_{231}(x) = \frac{1}{1 - (A_{231}(x) - x - x^2)},  \] 
where $A_{231}(x)=\sum\limits_{n=1}^\infty\,\abs{\Sif_n(231)} x^n$, and so
\begin{equation}\label{eq:S(x)T(x)connection}
 A_{231}(x) = 1+ x + x^2 - \frac{1}{1 + T_{231}(x)}.
\end{equation}

\begin{prop} \label{prop:tn_seq}
Let $n\ge 3$. If $f^{(k)}_n(231) = \abs{\{\sigma\in\mathcal{S}_n(231): \mathrm{fp}(\sigma)=k\}}$, then
\[ t_n(231) = \sum_{k=0}^{\floor{\frac{n}{2}}} (-1)^{k} f^{(k)}_{n-k}(231). \]
\end{prop}
\begin{proof}
For $n\ge 3$, let $\mathcal{A}_n(k, j)$ be the set of permutations in $\mathcal{S}_n(231)$ having exactly $k$ fixed points and $j$ \sif boxes of size 2. Every permutation in $\mathcal{A}_n(k, j)$ can be uniquely constructed from an element $\sigma\in \mathcal{A}_{n-j}(k+j, 0)$ by keeping $k$ of the fixed points of $\sigma$ and inflating each of the remaining $j$ points with $21$ boxes. In other words, we have
\[ \abs{\mathcal{A}_n(k,j)} = \binom{k+j}{k} \abs{\mathcal{A}_{n-j}(k+j,0)}. \]
This implies
\begin{align*}
  f^{(k)}_{n-k}(231) &= \sum_{j=0}^{\floor{\frac{n-2k}{2}}} \abs{\mathcal{A}_{n-k}(k,j)} \\
  &= \sum_{j}^{\floor{\frac{n-2k}{2}}} \binom{k+j}{k} \abs{\mathcal{A}_{n-k-j}(k+j,0)} = \sum_{\ell=k}^{\floor{\frac{n}{2}}} \binom{\ell}{k} \abs{\mathcal{A}_{n-\ell}(\ell,0)}.
\end{align*}
Therefore,
\begin{align*}
  \sum_{k=0}^{\floor{\frac{n}{2}}} (-1)^k f^{(k)}_{n-k}(231) 
  &= \sum_{k=0}^{\floor{\frac{n}{2}}} \sum_{\ell=k}^{\floor{\frac{n}{2}}} (-1)^k \binom{\ell}{k} \abs{\mathcal{A}_{n-\ell}(\ell,0)} \\
  &= \sum_{\ell=0}^{\floor{\frac{n}{2}}}\, \abs{\mathcal{A}_{n-\ell}(\ell,0)} \sum_{k=0}^{\ell} (-1)^k\binom{\ell}{k} = \abs{\mathcal{A}_{n}(0,0)},
\end{align*}
which proves the proposition since $\mathcal{T}_n(231) = \mathcal{A}_n(0,0)$.
\end{proof}

\begin{thm}
The function $A_{231}(x)=\sum\limits_{n=1}^\infty\,\abs{\Sif_n(231)} x^n$ is given by the continued fraction
\begin{equation}\label{eqn:SIF_contfrac}
	A_{231}(x) = \dfrac{x}{1+C_1x^2(x+1)-\dfrac{x}{1+C_2x^3(x+1)-\dfrac{x}{1+C_3x^4(x+1)-\dfrac{x}{\ddots}}}}.
\end{equation}
\end{thm}
\begin{proof}
If we let
\[ F_{231}(t,x) = \sum_{n\ge 0}\; \sum_{\pi\in\mathcal{S}_n(231)} t^{\mathrm{fp}(\pi)} x^n, \]
we can deduce from Elizalde's results \cite[Thm.~3.7]{Eli04} the following:
\begin{equation}\label{eqn:FIX_contfrac}
	F_{231}(t,x) = \dfrac{1}{1-(t-1)C_0 x-\dfrac{x}{1-(t-1)C_1x^2-\dfrac{x}{1-(t-1)C_2x^3-\dfrac{x}{\ddots}}}}.
\end{equation}
Moreover, by Proposition~\ref{prop:tn_seq}, we have
\[ 1+ T_{231}(x) = F_{231}(-x,x). \]
The claimed continued fraction for $A_{231}(x)$ then follows from \eqref{eq:S(x)T(x)connection}.
\end{proof}

\medskip
Suppose $\sigma\in\mathcal{S}_n(231)$ has a \sif box of size $3\le m<n$ that stabilizes the interval $[i,j]$. Since a \sif box of size greater than 2 must contain an ascent, there cannot be elements to the right of $\sigma(j)$ that are smaller than $i$. Thus, $\sigma$ stabilizes the interval $[j+1,n]$, and so it must also stabilize the interval $[1,i-1]$. In other words, $\sigma$ admits a sum decomposition of the form $\sigma = \sigma_1\oplus \beta\oplus \sigma_2$, where $\beta\in \Sif_m(231)$ and $\sigma_1,\sigma_2$ are $231$-avoiders, possibly empty.

As a consequence, the elements of $\Ind_n(231)\setminus\Sif_n(231)$ have no \sif box of size greater than 2. Moreover, they have to start with $n$ and are in one-to-one correspondence with ordered rooted trees with $n$ vertices, see \cite[Sec.~3]{BKP18}. We will use this correspondence to provide an alternative approach to derive \eqref{eqn:SIF_contfrac} and \eqref{eqn:FIX_contfrac}.

In fact, using the left-to-right maxima as roots, every element $\sigma\in \mathcal{S}_n(231)$ corresponds to a forest $T_\sigma$ of ordered trees constructed as follows: we declare that $\sigma(i)$ is the parent of $\sigma(j)$ if and only if it is a minimal inversion ($i<j$, $\sigma(i)>\sigma(j)$, and $j-i$ is minimal over all $i$). Then we order siblings in increasing order of labels from left to right. This can be visualized using the plot of the permutation or by means of an arc diagram, see the example in Figure~\ref{fig:231forest}.

\begin{figure}[h!]\centering
\begin{tikzpicture}[scale=0.4]
\begin{scope}
\draw[thick,cyan!40] (2,1)--(1,3)--(3,2);
\draw[thick,cyan!40] (6,5)--(5,9)--(7,8)--(9,6);
\plotpermutation{3,1,2,4,9,5,8,7,6}
\end{scope}
\begin{scope}[xshift=330, yshift=100]
\draw[cyan,thick] (0,0) parabola bend (0.75,0.3) (1.5,0);
\draw[cyan,thick] (0,0) parabola bend (1.5,0.6) (3,0);
\draw[cyan,thick] (6,0) parabola bend (6.75,0.3) (7.5,0);
\draw[cyan,thick] (6,0) parabola bend (7.5,0.6) (9,0) parabola bend (9.75,0.3) (10.5,0) parabola bend (11.25,0.3) (12,0);
\foreach \i/\v in {0/3,1/1,2/2,3/4,4/9,5/5,6/8,7/7,8/6}{
\draw[fill] (1.5*\i,0) circle(0.1);
\node[below=2pt] at (1.5*\i,0){\small $\v$};
}
\end{scope}
\end{tikzpicture}
\caption{Ordered forest corresponding to $\sigma=312495876$.}
\label{fig:231forest}
\end{figure}

For a vertex $v$ of $T_\sigma$, we let $\mathrm{descd}(v)$ be the number of descendants of $v$ and $\mathrm{depth}(v)$ be the number of ancestors, or depth of $v$. By construction,
\begin{align*}
 \mathrm{descd}(v) &= \#\{j\in[i+1,n]:\sigma(i)>\sigma(j)\}, \text{ and} \\
 \mathrm{depth}(v) &= \#\{j\in[i-1]:\sigma(j)>\sigma(i)\}. 
\end{align*}

Let $v_k$ be the vertex labeled $\sigma(k)$. Since there are $\sigma(k)-1$ entries smaller than $\sigma(k)$ and $k-1-\mathrm{depth}(v)$ of them appear at positions earlier than $k$, we must have $\sigma(k)-k+\mathrm{depth}(v)$ of them appearing after position $k$. Since this is the number of descendants of $v$, we get
\begin{equation*}
	\sigma(k)-k=\mathrm{descd}(v)-\mathrm{depth}(v).
\end{equation*} 
This formula gives us a way to identify fixed points $\sigma(i)=i$ and adjacent transpositions $\sigma(j)=j+1$, $\sigma(j+1)=j$ directly from the ordered forest.

We consider forests where ``depth" in the forest starts from roots of depth $p$ and define 
\begin{equation*}
F^{(p)}_{231}(t,s;x)=\sum_{n\ge 0}\sum_{T_n} t^{\mathrm{fp}(T_n)}s^{\mathrm{adt}(T_n)}x^n,
\end{equation*}
where the second sum is over all ordered forests with $n$ vertices and root depth $p$,
\begin{align*}
\mathrm{fp}(T) &:=\{v\in V(T):\mathrm{descd}(v)=\mathrm{depth}(v)\}\text{,~and~}\\ 
\mathrm{adt}(T)&:=\{(v,w)\in V(T):\mathrm{descd}(w)=\mathrm{depth}(v),\text{~and~}w\text{~is the only child of~}v\}.
\end{align*}

To enumerate this we define a larger set of ordered forests where the vertices in $\mathrm{fp}(T_n)$ can be marked in three ways, say \textit{good}, \textit{negative}, and \textit{positive} (weighted $tx$, $-x$, and $+x$), and the pairs in $\mathrm{adt}(T_n)$ can also be marked with weights $sx^2$, $-x^2$, and $+x^2$. All other vertices are weighted $+x$. It is then straightforward to create a sign-reversing involution by reversing the marking on the first positive or negative fixed point or adjacent pair, cancelling everything except the forests with vertices weighted $t$ and adjacent pairs weighted $s$. 

For any fixed point $v$ or adjacent pair $(v,w)$ in our forest, there cannot be another fixed point $v$ or adjacent pair $(v,w)$ as its descendant (since the depth increases and the number of descendants decreases). So, the children of such a point or pair make an ordered forest of size $\mathrm{depth}(v)$, of which there are $C_{\mathrm{depth}(v)}$. Then, $F^{(p)}_{231}(t,s;x)$ enumerates a sequence of trees, each of which is a good or negative fixed point or adjacent pair (weighted $tx$ or $-x$, or $sx^2$ or $-x^2$) followed by its children ($C_p$ possible ordered forests weighted $x^p$) or any kind of positive vertex weighted $x$ and its children (enumerated by $F^{(p+1)}_{231}(t,s;x)$). Thus, we have
\[ F^{(p)}_{231}(t,s;x)=\frac{1}{1-\big(x(t-1)x^pC_p + x^2(s-1)x^pC_p + xF^{(p+1)}_{231}(t,s;x)\big)}, \]
and so
\begin{equation*}
	F^{(0)}_{231}(t,s;x)=\frac{1}{1-x(t-1)C_0-x^2(s-1)C_0-\dfrac{x}{1-x^2(t-1)C_1-x^3(s-1)C_1-\dfrac{x}{\ddots}}}.
\end{equation*}
Now, by the above bijection, we have
\begin{equation*}
F^{(0)}_{231}(t,s;x)=\sum_{n\ge 0}\sum_{\;\pi\in S_n(231)} t^{\mathrm{fp}(\pi)}s^{\mathrm{adt}(\pi)}x^n,
\end{equation*} 
where $\mathrm{adt}(\pi)$ is the number of adjacent transpositions of $\pi$ (SIF boxes of size 2). In particular, $F^{(0)}_{231}(t,1;x)=F_{231}(t,x)$, and we recover Equation~\eqref{eqn:FIX_contfrac}.

Moreover, since the only \sif boxes for a permutation in $\Ind_n(231)\setminus\Sif_n(231)$ are fixed points or adjacent transpositions, the elements of $\Sif_n(231)$ correspond to trees which contain no fixed points or adjacent transpositions, except possibly at the root. These are trees with a root vertex weighted $x$ together with a forest beneath it beginning at depth 1 which contains no fixed point or adjacent transpositions. That is, we have $A_{231}(x)=xF^{(1)}_{231}(0,0;x)$, which gives us Equation~\eqref{eqn:SIF_contfrac}.

\section{Avoiding a pair of patterns of size 3}
\label{sec:pair_of_patterns}

With the same argument as in Theorem~\ref{thm:321avoiders}, we have $\Sif_n(321,\sigma) = \Ind_n(321,\sigma)$ for every pattern $\sigma$. Moreover, $\Sif_n(123, 321)=\varnothing$ for $n>4$, and for $n\ge 2$ we have 
\[ \Sif_n(231, 321)=\{1\ominus\id_{n-1}\} \;\text{ and }\; \Sif_n(132, 321)=\{\id_{n-j}\ominus\id_j: 1\le j<n\}, \]
where $\id_\ell=1\cdots \ell$ denotes the identity permutation of size $\ell$.
Recall that for $n>2$, every \sif permutation must contain a $231$ or a $312$ pattern. Therefore, $\Sif_n(231, 312)=\varnothing$ for $n>2$. The remaining pairs can be grouped in 4 Wilf equivalence classes (see Table~\ref{tab:PairOfPatterns}) that we proceed to discuss below.

\begin{table}[ht]
\small
\begin{tabular}{|c|l|c|c|} \hline
\R  $(\sigma_1,\sigma_2)\in \mathcal{S}_3\times\mathcal{S}_3$ & \hspace{40pt} $\abs{\Sif_n(\sigma_1,\sigma_2)}$, $n\ge 1$ & OEIS & Section \\[2pt] \hline
\rule[-2.8ex]{0ex}{7ex} \parbox{11ex}{(123, 132)\\[1pt] (123, 213)} &1, 1, 2, 3, 6, 9, 18, 27, 54, 81, 162, 243, \dots & A182522 & 3.2 \\ \hline
\rule[-2.8ex]{0ex}{7ex} \parbox{11ex}{(123, 231)\\[1pt] (123, 312)} &1, 1, 1, 1, 2, 3, 3, 5, 5, 7, 7, 10, 9, 13, 12, \dots & A363433 & 3.4 \\ \hline
\R (123, 321) &1, 1, 2, 3, 0, 0, 0, 0, 0, 0, 0, 0, \dots & finite & \\ \hline
\R (132, 213) &1, 1, 2, 5, 8, 17, 26, 53, 80, 161, 242, 485, \dots & A62318 & 3.3 \\ \hline
\rule[-5.7ex]{0ex}{12.8ex} \parbox{11ex}{(132, 231)\\[1pt] (132, 312)\\[1pt] (213, 231)\\[1pt] (213, 312)} &1, 1, 1, 2, 4, 8, 16, 32, 64, 128, 256, 512, \dots & A011782 & 3.1 \\ \hline
\rule[-2.8ex]{0ex}{7ex} \parbox{11ex}{(132, 321)\\[1pt] (213, 321)} &1, 1, 2, 3, 4, 5, 6, 7, 8, 9, 10, 11, \dots & A028310 & \\ \hline
\R (231, 312) & 1, 1, 0, 0, 0, 0, 0, 0, 0, 0, 0, 0, \dots &  finite & \\ \hline
\rule[-2.8ex]{0ex}{7ex} \parbox{11ex}{(231, 321)\\[1pt] (312, 321)} & 1, 1, 1, 1, 1, 1, 1, 1, 1, 1, 1, 1, \dots &  A000012 & \\ \hline
\end{tabular}
\bigskip
\caption{$(\sigma_1,\sigma_2)$-avoiding \sif permutations.}
\label{tab:PairOfPatterns}
\end{table}

\subsection{Enumeration of $\Sif_n(132,231)\sim \Sif_n(213,312)$}
We begin by observing that every $\sigma\in\Sif_n(132,231)$ must have $\sigma(1)=n$. Otherwise, $\sigma$ would be either decomposable or it would contain a $132$ or a $231$ pattern. With a similar argument, one can see that the plot of $\sigma$ must have a $\bigvee$ shape, i.e., it consists of a decreasing sequence (starting with $n$) followed by an increasing sequence. For example, for $n=6$ we have:
\begin{gather*}
654123,\; 654213,\; 651234,\; 652134, \\
641235,\; 642135,\; 612345,\; 631245.
\end{gather*}
Also note that for every $\sigma\in\Sif_n(132,231)$, either $\sigma(2) = n-1$ or $\sigma(n) = n-1$.
\begin{thm}
For $n\ge 2$, we have $\abs{\Sif_n(132,231)} = 2^{n-2}$.
\end{thm}
\begin{proof}
For $n\ge 3$, let $\mathcal{A}_n^{(k)} = \{\sigma\in\Sif_n(132,231): \sigma(k)=n-1\}$. Then,
\[ \Sif_n(132,231) = \mathcal{A}_n^{(2)}\cup \mathcal{A}_n^{(n)}. \]
For $\sigma\in \mathcal{A}_n^{(n)}$, let $\hat \sigma$ be the permutation defined by $\hat\sigma(1)=n-1$ and $\hat\sigma(j)=\sigma(j)$ for $j\in\{2,\dots,n-1\}$. The map $\sigma\mapsto \hat\sigma$ is a bijection from $\mathcal{A}_n^{(n)}$ to $\Sif_{n-1}(132,231)$.

For $\sigma\in \mathcal{A}_n^{(2)}$ we let $\sigma'$ be the permutation obtained by removing $n$ from $\sigma$. That is, $\sigma'(j)=\sigma(j+1)$ for $j\in\{1,\dots,n-1\}$. Clearly, $\sigma'$ begins with $n-1$ and its plot also has a $\bigvee$ shape. Thus, $\sigma'$ avoids the patterns $132$ and $231$, and it cannot have a \sif box of size larger than 2 and smaller than $n-1$. Also, since $\sigma$ has no fixed points, there is no $j$ such that $\sigma'(j-1)=j$, which means that $\sigma'$ has no \sif box of size 2. In other words, $\sigma'$ is \sif if and only if it is fixed point free. In addition, $\sigma'$ cannot have a fixed point $\sigma'(j)=j$ on the right side of the $\bigvee$ (since the $n-1-j$ entries past $j$ will have to be greater than $j$, but $n-1$ is already at position 1). In other words, $\sigma'$ has at most one fixed point. 

With the above observations in mind, we let
\begin{equation*}
\phi(\sigma) = \begin{cases}
	\sigma' &\text{if } \mathrm{fp}(\sigma')=0, \\
	\tau_{j}\circ\sigma' &\text{if } \sigma'(j)=j,
\end{cases}
\end{equation*}
where $\tau_{j}$ denotes the transposition that swaps $j$ and $j+1$. By construction, $\phi$ maps $\mathcal{A}_n^{(2)}$ into $\Sif_{n-1}(132,231)$ (note that for $\sigma\in\mathcal{A}_n^{(2)}$, we must have $\sigma'(j)\not= j+1$ for every $j$). Moreover, for $\varrho\in\Sif_{n-1}(132,231)$, we have $\phi^{-1}(\varrho) = 1\ominus\varrho$ if $ \mathrm{fp}(1\ominus\varrho)=0$, and $\phi^{-1}(\varrho) = 1\ominus(\tau_j\circ\varrho)$ if $\varrho(j)=j+1$. Combining the above bijections for $\mathcal{A}_n^{(n)}$ and $\mathcal{A}_n^{(2)}$, we conclude
\[ \abs{\Sif_n(132,231)} = \abs{\mathcal{A}_n^{(2)}}+\abs{\mathcal{A}_n^{(n)}} = 2\,\abs{\Sif_{n-1}(132,231)}, \]
which leads to the claimed formula. 
\end{proof}

\subsection{Enumeration of $\Sif_n(123,132)\sim \Sif_n(123,213)$}
Since the permutations $123$ and $132$ are both not SIF, we have $\Sif_n(123,132)= \Sif_n$ for $n\in\{1,2,3\}$. 

For $n\ge 4$ and every $\sigma\in \Sif_n(123,132)$, the pattern avoidance implies $\sigma(1)\in\{n-1,n\}$ and $\sigma^{-1}(1)\in\{n-1,n\}$. On the other hand, the \sif property implies that we cannot have $\sigma(1)=n$ and $\sigma(n)=1$. Therefore, the set $\Sif_n(123,132)$ is the union of the disjoint sets
\begin{align*}
 \mathcal{A}_{1,n} &= \{\sigma\in \Sif_n(123,132): \sigma(1)=n \text{ and } \sigma(n-1)=1\}, \\
 \mathcal{A}_{2,n} &= \{\sigma\in \Sif_n(123,132): \sigma(1)=n-1 \text{ and } \sigma(n-1)=1\}, \\
 \mathcal{A}_{3,n} &= \{\sigma\in \Sif_n(123,132): \sigma(1)=n-1 \text{ and } \sigma(n)=1\}.
\end{align*} 
Observe that the map $\sigma\mapsto \sigma^{-1}$ gives a bijection between $\mathcal{A}_{1,n}$ and $\mathcal{A}_{3,n}$. Moreover, if $\tau_{n-1}$ is the transposition that swaps $n-1$ and $n$, then the map $\sigma\mapsto \sigma\circ\tau_{n-1}$ gives a bijection between $\mathcal{A}_{2,n}$ and $\mathcal{A}_{3,n}$. Thus, for $n\ge 4$, we have $\abs{\mathcal{A}_{1,n}} = \abs{\mathcal{A}_{2,n}} = \abs{\mathcal{A}_{3,n}}$.

\begin{thm}
The sequence $a_n = \abs{\Sif_n(123,132)}$ satisfies the recurrence relation
\begin{equation*}  
a_1=a_2=1, \; a_3=2, \text{ and } a_n=3a_{n-2} \;\text{ for } n\ge 4,
\end{equation*}
and its generating function satisfies $A_{123,132}(x) = \dfrac{x+x^2-x^3}{1-3x^2}$. 
\end{thm}
\begin{proof}
For $n\ge 4$ and $\sigma\in \mathcal{A}_{1,n}$, we let $\sigma''$ be the permutation obtained from $\sigma$ by removing $n$ and $1$, and then reducing all of the remaining entries by 1. For example, for $n=6$ the map $\sigma\mapsto \sigma''$ gives:
\begin{equation*}
654213 \mapsto 4312, \quad
645213 \mapsto 3412, \quad
645312 \mapsto 3421.
\end{equation*}
Observe that $\sigma\in \mathcal{A}_{1,n}$ implies $\sigma(n)\not=n-1$ (otherwise $\sigma$ would have the stabilized interval $[2, n-2]$). Moreover, removing the first element and the 1 value from $\sigma$ (i.e.\ shifting its plot left and down by one) guarantees that $\sigma''$ will not have any new stabilized sub-intervals on the interval $[2,n-2]$. Therefore, $\sigma''$ is \sif and the map $\sigma\mapsto \sigma'': \mathcal{A}_{1,n}\to \Sif_{n-2}(123,132)$ is clearly bijective. Finally, since $\Sif_{n}(123,132) = \mathcal{A}_{1,n} \cup \mathcal{A}_{2,n}\cup \mathcal{A}_{3,n}$ is a disjoint union, and since $a_{n-2} = \abs{\mathcal{A}_{1,n}} = \abs{\mathcal{A}_{2,n}} = \abs{\mathcal{A}_{3,n}}$, we conclude 
\[ a_n = \abs{\mathcal{A}_{1,n}} + \abs{\mathcal{A}_{2,n}} + \abs{\mathcal{A}_{3,n}} = 3a_{n-2} \;\text{ for } n\ge 4. \]
As a consequence, $A_{123,132}(x) = x+x^2+2x^3 + 3x^2(A_{123,132}(x)-x)$.
\end{proof}

\subsection{Enumeration of $\Sif_n(132,213)$}
Every permutation $\sigma\in\mathcal{S}_n(132,213)$ must be of the form $\sigma = \id_{\ell_1}\ominus\cdots\ominus\id_{\ell_k}$, where $\id_\ell=1\cdots \ell$ and $\ell_1+\dots+\ell_k=n$. Clearly, the permutations $\id_n$ (case $k=1$) and its reverse $r(\id_n)$ (case $k=n$) are not SIF, so we can assume $1<k<n$. Moreover, if $i_1<\cdots<i_{k-1}$ are the positions of the descents\footnote{A permutation $\sigma$ is said to have a descent at position $i$ if $\sigma(i)>\sigma(i+1)$.} of $\sigma$, then $\sigma$ is \sif if and only if $i_j + i_\ell\not =n$ for every $j\not=\ell$.

\begin{thm}
For $n\ge 2$, we have 
\[ \abs{\Sif_n(132,213)} = 
\begin{cases} 
3^{\frac{n-1}{2}}-1 &\text{if $n$ is odd}, \\
2\cdot 3^{(\frac{n}{2}-1)}-1&\text{if $n$ is even}.
\end{cases} \]
Its generating function satisfies $A_{132,213}(x) = \dfrac{x-2x^3+3x^4}{(1-x)(1-3x^2)}$. 
\end{thm}
\begin{proof}
As explained above, the elements of $\Sif_n(132,213)$ are determined by the positions of their descents. More precisely, if \[ \mathcal{B}_{n,k} = \big\{(i_1,\dots,i_{k-1})\in \mathbb{N}^{k-1}:  i_1<\cdots<i_{k-1}<n \text{ and } i_j+i_\ell \not=n \text{ for } j\not=\ell\big\}, \]
then $\abs{\Sif_n(132,213)} = \sum\limits_{k=2}^{n-1} \abs{\mathcal{B}_{n,k}}$. Suppose $n=2m+1$. Consider the nested arc diagram with $2m$ nodes, labeled $1$ through $2m$, and such that each node $j\le m$ is connected by a single arc to the node $2m+1-j$. Then, each element of $\mathcal{B}_{n,k}$ can be obtained by choosing $k-1$ arcs out of the $m$ arcs in the above arc diagram, and then for each arc, choosing one of the two nodes in that arc. Therefore $\abs{\mathcal{B}_{2m+1,k}} = \binom{m}{k-1} 2^{k-1}$, and we get
\[ \abs{\Sif_{2m+1}(132,213)} = \sum_{k=2}^{m+1} \binom{m}{k-1} 2^{k-1} = \sum_{k=1}^{m} \binom{m}{k} 2^{k} = 3^m-1.\]
If $n=2m$, we consider a similar arc diagram with $2m-1$ nodes, pairing each $j<m$ with the node $2m-j$, and leaving the node $m$ alone. There are $\binom{m-1}{k-2} 2^{k-2}$ elements of $\mathcal{B}_{2m,k}$ that include $m$ and 
$\binom{m-1}{k-1} 2^{k-1}$ elements that do not include $m$. Therefore,
\begin{align*}
 \abs{\Sif_{2m}(132,213)} &= \sum_{k=2}^{m+1} \binom{m-1}{k-2} 2^{k-2} + \sum_{k=2}^{m} \binom{m-1}{k-1} 2^{k-1} \\
 &= 3^{m-1} + 3^{m-1} - 1 = 2\cdot 3^{m-1} - 1.
\end{align*}
The generating function follows using routine algebraic manipulations. 
\end{proof}

\subsection{Enumeration of $\Sif_n(123,231)\sim \Sif_n(123,312)$}
As we did for $\Sif_n(123)$ and $\Sif_n(231)$, we will rely on the fixed points statistics provided in \cite{Eli04}. If we let
\[ F_{123,231}(t,x) = \sum_{n\ge 0}\; \sum_{\pi\in\mathcal{S}_n(123,231)} t^{\mathrm{fp}(\pi)} x^n, \]
then we have \cite[Cor.~4.13]{Eli04}: 
\begin{equation*}
 F_{123,231}(t,x) = \frac{1 + tx + (t^2-2)x^2 + (3+t-t^2)x^4 + 3x^5 - x^6 - 4x^7 - 2tx^8}{(1-x^2)^3(1-x^3)}.
\end{equation*}

\begin{thm}
For $n\ge 3$, we have
\begin{equation*}
  \abs{\Sif_n(123,231)} = f_n^{(0)}(123,231) - f_{n-1}^{(1)}(123,231) + f_{n-2}^{(2)}(123,231),
\end{equation*}
and its generating function $A_{123,231}(x)$ satisfies
\[ A_{123,231}(x) - x - x^2 = F_{123,231}(-x,x) - 1 = \frac{x^3(1+x-x^4)}{(1-x^2)^2(1-x^3)}. \]
\end{thm}
\begin{proof}
As in the proof of Theorem~\ref{thm:123avoiders}, we focus on counting fixed-point-free permutations that are not SIF. With a similar argument as for Lemma~\ref{lem:2box}, one can deduce that there are $f_{n-1}^{(1)}(123,231) - f_{n-2}^{(2)}(123,231)$ fixed-point-free permutations in $\mathcal{S}_n(123,231)$ having a smallest \sif box of size $2$. 

For $n=3$, the statement is clearly true since $f_3^{(0)}=1$ and $f_{2}^{(1)}=f_{1}^{(2)}=0$.  For $n>3$, a permutation $\sigma\in\mathcal{S}_n(123,231)$ cannot have a \sif box of size $m$ with $3\le m<n$. In fact, the avoidance of $231$ forces $\sigma$ to be decomposable, but the $123$-avoidance does not allow it. In conclusion, every fixed-point-free permutation in $\mathcal{S}_n(123,231)$ is either \sif or it has a \sif box of size 2. So, $f_n^{(0)}(123,231) = \abs{\Sif_n(123,231)} + f_{n-1}^{(1)}(123,231) - f_{n-2}^{(2)}(123,231)$.
\end{proof}

\begin{rem}
For $n\ge 3$, we have
\begin{equation*}
\abs{\Sif_n(123,231)} = 
\begin{cases}
\frac{(n+12)(n-2)}{24} &\text{if } n \equiv 0 \mod 6, \\
\frac{(n+5)(n-1)}{24} &\text{if } n \equiv 1 \mod 6, \\
\frac{(n+12)(n-2)}{24} &\text{if } n \equiv 2 \mod 6, \\
\frac{(n+3)(n+1)}{24} &\text{if } n \equiv 3 \mod 6, \\
\frac{(n+12)(n-2)}{24}-\frac13 &\text{if } n \equiv 4 \mod 6, \\
\frac{(n+3)(n+1)}{24} &\text{if } n \equiv 5 \mod 6.  
\end{cases}
\end{equation*}
\end{rem}

\renewcommand\thesection{\Alph{section}}
\setcounter{section}{0}

\section{Appendix}

Let $f^{(2)}_n(123) = \abs{\{\sigma\in\mathcal{S}_n(123): \mathrm{fp}(\sigma)=2\}}$. Using a bijection to Dyck paths and tracking several parameters at the Dyck paths level, Elizalde proved an elaborate formula for $f^{(2)}_n(123)$, see \cite[Theorem~3.3]{Eli04}. As an alternative approach, we consider the sets 
\[ \mathcal{S}_{n,d}^{(2)}(123)=\big\{\sigma\in \mathcal{S}_n(123): \exists\; i,j \text{ with } \sigma(i)=i,\: \sigma(j)=j,\text{ and } |i-j|=d\big\}. \] 
Clearly, if we let $f_{n,d}^{(2)}(123)=\abs{\mathcal{S}_{n,d}^{(2)}(123)}$, then
\begin{equation}\label{eq:twoFP}
  f^{(2)}_n(123) = \sum_{d\ge 1} f_{n,d}^{(2)}(123).
\end{equation}

\begin{lem}
For $n\ge 2$, we have $f_{n,d}^{(2)}(123)=0$ unless $1\le d\le \tfrac{n}{2}$.
\end{lem}
\begin{proof}
The argument is simple with a picture. The plot of any permutation in $\mathcal{S}_{n,d}^{(2)}(123)$ must be of the form
\begin{equation}\label{pic:twoFP}
\begin{tikzpicture}[scale=0.6,baseline=10ex]
\node[below=1] at (2,0) {\small $i$};
\node[below=1] at (4,0) {\small $i+d$};
\node at (1,5) {$j$ {\small dots}};
\node at (5.2,1) {$k$ {\small dots}};
\node[gray] at (3,5.2) {$\ddots$};
\node[gray] at (1,3.2) {$\ddots$};
\node[gray] at (5,3.2) {$\ddots$};
\node[gray] at (3,1.2) {$\ddots$};
\clip (0.1,0.1) rectangle (5.9,5.9);
\draw[gray,step=2] (0,0) grid (6,6);
\draw[mesh] (0,0) rectangle (2,2);
\draw[mesh] (2,2) rectangle (4,4);
\draw[mesh] (4,4) rectangle (6,6);
\foreach \x/\y in {2/2,4/4}{\draw[fill] (\x,\y) circle (0.12);}
\end{tikzpicture}
\end{equation}
with $j,k\ge 0$, where the dots on the upper left corner and lower right corner must avoid a $123$ pattern, and each symbol $\ddots$ stands for a decreasing sequence, possibly empty. Now, there are $d-1$ points in the middle row and $d-1$ points in the middle column, so the total number of points is $n = j + k + 2(d-1) + 2 = j+k+2d$. Therefore, $2d = n-j-k \le n$ and so $d\le \frac{n}{2}$.
\end{proof}

Let us discuss the extreme cases $d=1$ or $d=\floor{\tfrac{n}{2}}$.

If $d=1$, we must have $j=k$ in \eqref{pic:twoFP} and the permutation must be of even size. Thus,
\begin{equation*}
  f^{(2)}_{n,1}(123) =
  \begin{cases}
  C_{\frac{n-2}{2}}^2 &\text{if $n$ is even,} \\
  0  &\text{if $n$ is odd.}
  \end{cases} 
\end{equation*}

On the other hand, if $d=\floor{\tfrac{n}{2}}$, we have
\begin{equation*}
  f^{(2)}_{n,\floor{\tfrac{n}{2}}}(123) =
  \begin{cases}
   \dbinom{n-2}{\tfrac{n-2}{2}} &\text{if $n$ is even,} \\[2ex]
   (n-3)\dbinom{n-3}{\tfrac{n-1}{2}} &\text{if $n$ is odd.}
  \end{cases} 
\end{equation*}
Indeed, if $n$ is even, we have $d=\frac{n}{2}$ and \eqref{pic:twoFP} takes the form

\medskip
\begin{center}
\begin{tikzpicture}[scale=0.6]
\node[below=1] at (2,0) {\small $i$};
\node[below=1] at (4,0) {\small $i+\frac{n}{2}$};
\node[gray] at (3,5.2) {$\ddots$};
\node[gray] at (1,3.2) {$\ddots$};
\node[gray] at (5,3.2) {$\ddots$};
\node[gray] at (3,1.2) {$\ddots$};
\clip (0.1,0.1) rectangle (5.9,5.9);
\draw[gray,step=2] (0,0) grid (6,6);
\draw[mesh] (0,4) rectangle (2,6);
\draw[mesh] (0,0) rectangle (2,2);
\draw[mesh] (2,2) rectangle (4,4);
\draw[mesh] (4,4) rectangle (6,6);
\draw[mesh] (4,0) rectangle (6,2);
\foreach \x/\y in {2/2,4/4}{\draw[fill] (\x,\y) circle (0.12);}
\end{tikzpicture}
\end{center}
Now, there are $\binom{\frac{n-2}{2}}{i-1}$ ways to choose the elements to the left of $i$ and $\binom{\frac{n-2}{2}}{i-1}$ ways to choose the positions of the elements smaller than $i$. Thus,
\[ f^{(2)}_{n,\tfrac{n}{2}}(123) = \sum_{i=1}^{\frac{n}{2}} \binom{\frac{n-2}{2}}{i-1} \binom{\frac{n-2}{2}}{i-1} = \dbinom{n-2}{\tfrac{n-2}{2}}. \]

If $n$ is odd, then $d=\frac{n-1}{2}$ and in \eqref{pic:twoFP} we must have either $j=1$ and $k=0$, or $j=0$ and $k=1$. In both cases, we have $i\ge 2$, and a similar counting argument gives
\[ f^{(2)}_{n,\tfrac{n-1}{2}}(123) = (n-3)\sum_{i=2}^{\frac{n-1}{2}} \binom{\frac{n-3}{2}}{i-2} \binom{\frac{n-3}{2}}{i-1} = (n-3)\dbinom{n-3}{\tfrac{n-1}{2}}. \]

In general, we claim the following.
\begin{conjecture}
Write $n=2m+r$ with $r\in\{0,1\}$ and let $1\le d\le m$. Then, 
\begin{equation*}
 f_{n,d}^{(2)}(123) = \frac{\binom{n-d}{m}^2}{(n-d)^2} \Bigg[\frac{\binom{2d}{d}d^3(1-r)}{4d-2}
  + 2\!\sum_{i=1}^{\floor{\frac{m+1}{3}}}\frac{\binom{m}{i-r}\binom{m-d+r}{i}\binom{2d-2}{2i-1+d-r}}{\binom{m+i}{m}\binom{m-d+i}{m-d+r}}(d^2-(2i-r)^2)\Bigg].
\end{equation*}
\end{conjecture}

\begin{rem}
Combining \eqref{eq:twoFP} with the above conjecture, we get an expression for $f^{(2)}_n(123)$ that is simpler and more efficient than the existing ones. When implemented in Sage, \eqref{eq:twoFP} computed the first 50 terms of $f^{(2)}_n(123)$ in less than 0.8 seconds, and the first 100 terms in less than 6 seconds. On the other hand, the formula from \cite[Theorem~3.3]{Eli04} took more than 45 minutes to find the first 50 terms.
\end{rem}



\begin{thebibliography}{99}
\bibitem{ARW16} F.~Ardila, F.~Rinc\'on, L.~Williams, Positroids and non-crossing partitions, {\em Trans. Amer. Math. Soc.} \textbf{368} (2016), no. 1, 337--363.
%
\bibitem{BKP18} J.-L.~Baril, S.~Kirgizov, A.~Petrossian, Forests and pattern-avoiding permutations modulo pure descents, {\em Pure Math. Appl. (PU.M.A.)} \textbf{27} (2018), no. 1, 18--31.
%
\bibitem{NB14} N.~Blitvi{\'c}, Stabilized-interval-free permutations and chord-connected permutations, {\em Discrete Math. Theor. Comput. Sci. Proc.} {\bf AT}, 2014, 801-814.
%
\bibitem{Callan04} D.~Callan, Counting stabilized-interval-free permutations, {\em J. Integer Seq.} \textbf{7} (2004), Article 04.1.8.
%
\bibitem{Comtet72} L.~Comtet, Sur les coefficients de l'inverse de la s{\'e}rie formelle $\sum n!t^n$,  {\em C. R. Acad. Sci. Paris S{\'e}r. A-B} \textbf{275} (1972), A569--A572.
%
\bibitem{Eli04} S.~Elizalde, Multiple pattern avoidance with respect to fixed points and excedances, {\em Electron. J. Combin.} \textbf{11} (2004), \#R51, 40 pp.
%
\bibitem{GKZ16} A.~L.~L.~Gao, S.~Kitaev, P.~B.~Zhang, On pattern avoiding indecomposable permutations, {\em Integers} \textbf{18} (2018), Paper No. A2, 23 pp.
%
\bibitem{Kitaev11} S.~Kitaev, \emph{Patterns in permutations and words}, Monographs in Theoretical Computer Science, an EATCS Series, Springer, Heidelberg, 2011.
%
\bibitem{RSZ03} A. Robertson, D. Saracino, D. Zeilberger, Refined Restricted Permutations, {\em Ann. Comb.} \textbf{6} (2003), 427--444.
%
\bibitem{SiSch85} R.~Simion, F.~W.~Schmidt, Restricted permutations, {\em Europ. J. Combin.} \textbf{6} (1985), 383--406.
%
\bibitem{oeis} N. J. A.~Sloane, The On-Line Encyclopedia of Integer Sequences, \url{http://oeis.org}.
%
\end{thebibliography}
\end{document}